\documentclass[12pt]{amsart}

\usepackage{graphicx,color}
\usepackage{setspace}
\usepackage{amsmath, amssymb,amsthm, amsfonts}

\usepackage{accents}

\newtheorem{thm}{Theorem}[section]
\newtheorem{lma}{Lemma}[section]
\newtheorem{prop}{Proposition}[section]

\theoremstyle{definition}

\theoremstyle{remark}
\newtheorem{remark}{Remark}[section]

\numberwithin{equation}{section}

\newcommand{\tr}{\mbox{tr}}

\renewcommand{\div}{\mbox{div}}

\newcommand{\Ric}{\mbox{Ric}}
\newcommand{\R}{\mathbb R}

\newcommand{\be}{\begin{equation}}
\newcommand{\ee}{\end{equation}}
\newcommand{\bee}{\begin{equation*}}
\newcommand{\eee}{\end{equation*}}

\def\p{\partial}

\def\la{\langle}
\def\ra{\rangle}
\def\lf{\left}
\def\ri{\right}
\def\Pi{\displaystyle{\mathbb{II}}}

\def\S{\Sigma}

\def\vh{\vspace{.2cm}}

\def\l{\lambda}

\def\a{\alpha}

\def\ringA{\accentset{\circ}{A}}

\def\dggz{d \gamma_{g_0}}

\begin{document}
\title[]
{Some functionals on compact manifolds with boundary}

\author{Pengzi Miao$^1$}
\address[Pengzi Miao]{Department of Mathematics, University of Miami, Coral Gables, FL 33146, USA.}
\email{pengzim@math.miami.edu}
\thanks{$^1$Research partially supported by Simons Foundation Collaboration Grant for Mathematicians \#281105.}

\author{Luen-Fai Tam$^2$}
\address[Luen-Fai Tam]{The Institute of Mathematical Sciences and Department of
 Mathematics, The Chinese University of Hong Kong, Shatin, Hong Kong, China.}
 \email{lftam@math.cuhk.edu.hk}
\thanks{$^2$Research partially supported by Hong Kong RGC General Research Fund \#CUHK 14305114}

\renewcommand{\subjclassname}{
  \textup{2010} Mathematics Subject Classification}

\date{February, 2016}

\begin{abstract}
We analyze critical points of two functionals of Riemannian metrics on compact manifolds with boundary.
These functionals are motivated by
formulae of the mass functionals
 of asymptotically flat and asymptotically
hyperbolic manifolds
\end{abstract}


\maketitle

\markboth{Pengzi Miao and Luen-Fai Tam}{Some functionals on compact manifolds with boundary}

\section{introduction}\label{s-intro}
In this work we discuss critical  points
of two functionals defined on the space of  Riemannian metrics  on a compact manifold with  boundary.
Let $\Omega$ be an $n$-dimensional, connected, compact  manifold with smooth  boundary $ \Sigma$
   which itself is not necessarily connected.
Given a   Riemannian metric $g$ on $\Omega$,  let
\be\label{e-Gg}
G^g_\lambda :=\Ric(g)-\frac12\lf[ \mathcal{S}_g-\lambda(n-1)(n-2)\ri]g ,
\ee
where $\Ric(g)$, $\mathcal{S}_g$  are the Ricci curvature,  the scalar curvature of $g$, respectively,
and $\lambda \in \{ 0, 1 , -1\}$.
Let $X$ be a vector field on $\Omega$.
 Define
\be\label{e-F}
\mathcal{F}_X (g):=\int_\S G^g_\lambda  (X,\nu_g)d\gamma_g,
\ee
where  $\nu_g$ is the outward unit normal to $ \Sigma$ in $(\Omega, g)$ and $ d \gamma_g$ denotes  the volume element
 on  $\Sigma$ induced by $g$.

We will be interested in  $ \mathcal{F}_X (\cdot)$ constrained to  the space of constant scalar curvature metrics on $\Omega$.
To be precise,  we assume that  $g_0$ is a  given, smooth metric on $ \Omega$ such that $\mathcal{S}_{g_0} = \l n(n-1)$.
If $\lambda > 0$, we assume that the first Dirichlet eigenvalue of $(n-1)\Delta_{g_0}+\mathcal{S}_{g_0}$ is positive.
Let $ \gamma$ be  the metric on $ \Sigma$ induced from $ g_0$.
Define
\bee \label{e-spaces}
   \begin{split}
     \mathcal{M}^\lambda:=  & \{\text{$g$ is a smooth Riemannian metric on  $\Omega$  with $\mathcal{S}_{g} = \lambda n(n-1)$}\},  \\
     \mathcal{M}_\gamma^\lambda:=&\{\text{$g\in \mathcal{M}^\lambda $  such that $g|_{T(\Sigma)}= g_0 | _{T (\Sigma)} $}\}, \\
       \mathcal{M}^\lambda_0:=&\{\text{$g\in \mathcal{M}^\lambda $  such that $g=g_0$ everywhere at  $\Sigma$}\}.
 \end{split}
 \eee
Here, in the definition of $ \mathcal{M}^\lambda_\gamma$, $ T (\Sigma)$ denotes the tangent bundle of $ \Sigma$.
Evidently,  $g_0 \in \mathcal{M}^\lambda$ and $ \mathcal{M}^\lambda_0\subset \mathcal{M}^\lambda_\gamma \subset \mathcal{M}^\lambda$.

To state our result, we will assume  $ X$ is a {\em conformal Killing} vector field on $(\Omega, g_0)$, see {\bf (a2)} in section \ref{s-conformal}.
Under this assumption, it is known (cf. \cite{Herzlich}) that
$$ \text{$g_0$ is an Einstein metric}  \Rightarrow  \div_{g_0} X  \in \mathrm{Ker}  (  D \mathcal{S}_{g_0}^* ) ,$$
where ``$\div_{g_0} $"  is the divergence with respect to $g_0$ and
$ \mathrm{Ker}  (  D \mathcal{S}_{g_0}^* ) $ denotes the kernel of the operator $  D \mathcal{S}_{g_0}^* (\cdot)$,
which is  the formal $L^2$-adjoint of the linearization of the scalar curvature map
$\mathcal{S}$ at $g_0$, i.e.
$$ D \mathcal{S}_{g_0}^* (f) = - (\Delta_{g_0} f )g_0 + \nabla^2_{g_0} f -  f \Ric(g_0)  $$
for any  function $f$.
In a relativistic context, a nontrivial solution $f$ to  $ D \mathcal{S}_{g_0}^* (f) = 0 $ is often
referred as a static potential.
Motivated by the work of Corvino \cite{Corvino} on localized scalar curvature deformation,
we seek a variational  characterization of
 the condition  ``$\div_{g_0} X  \in \mathrm{Ker}  (  D \mathcal{S}_{g_0}^* )$",
 which is part of  the following theorem.

\begin{thm}\label{t-main-1}
Let $(\Omega,g_0)$,  $\mathcal{M}^\lambda$,   $ \mathcal{M}^\lambda_\gamma $ and $  \mathcal{M}^\lambda_0$
be described as above.
Suppose
 $X$ is a conformal Killing vector field on $(\Omega, g_0)$.
 Let $ \{ g(t) \}_{| t | < \epsilon  }$ be a smooth path of metrics on $\Omega$ such that $ g(0) = g_0$.
 Then
 \begin{enumerate}
   \item [(i)]
   $ \displaystyle{ \frac{d}{dt}\mathcal{F}_X(g(t))|_{t=0}=0} $
   for all path $\{ g(t) \} \subset   \mathcal{M}_0^\lambda$ if and only if $\div_{g_0}X\in \mathrm{Ker}(D\mathcal{S}^*_{g_0} (\cdot) )$.

   \item [(ii)]    $ \displaystyle{ \frac{d}{dt}\mathcal{F}_X (g(t))|_{t=0}=0} $
   for all  path $\{ g(t) \} \subset \mathcal{M}_\gamma^\lambda$ if and only if $\div_{g_0}X\in \mathrm{Ker}(D\mathcal{S}^*_{g_0} (\cdot) )$
   and $\nu_{g_0}$ is an eigenvector of  $\la X,\nu_{g_0}\ra G^{g_0}$ with zero eigenvalue at every point on $\Sigma$.

   \item [(iii)] If  $g_0$ is Einstein, then $ \displaystyle{ \frac{d}{dt}\mathcal{F}_X (g(t))|_{t=0}=0} $
   for all path $ \{ g(t) \} \subset  \mathcal{M}^\lambda$. The converse is true if
    there is a function $u$ such that
       $ \div_{g_0} X - X(u) $ has a fixed  sign on a dense subset
    in $\Omega$, which is satisfied if
 $ \div_{g_0} X$ does not vanish  on $\Omega$.
 \end{enumerate}
 \end{thm}

The use  of the functional $\mathcal{F}_X (\cdot) $ in Theorem \ref{t-main-1} is largely motivated by
corresponding expressions  of the  mass  functionals  of  asymptotically flat   and  asymptotically hyperbolic manifolds.
On an asymptotically flat manifold $(M^n, g)$, the ADM mass \cite{ADM61}  $\mathfrak{m}_{_{ADM}}$
can be computed by
\be \label{eq-mass-AF}
\mathfrak{m}_{_{ADM}}=  c_n \lim_{r\to\infty}\int_{S_r}G^g_0(X,\nu_g) d\gamma ,
\ee
where $c_n < 0$ is a dimensional constant, $S_r$ denotes the  coordinate sphere of radius $r$
in the asymptotically flat end of $(M, g)$,  and
$$
X = \sum_{i=1}^n x^i\frac{\p}{\p x^i}
$$
is the conformal Killing vector field on the  Euclidean space $\R^n$.
For the origin and the proof of \eqref{eq-mass-AF}, we refer readers to \cite{Ashtekar-Hansen, Chrusciel,
Huang-ICCM, MiaoTam2015, Herzlich, WangWu2015}.
On an asymptotically hyperbolic manifold, formulae analogous to \eqref{eq-mass-AF},
which compute the mass functionals introduced in \cite{Chrusciel-Herzlich},  can be found in \cite{Herzlich}.

The second functional that we are interested in is also related to  \eqref{eq-mass-AF}.
Indeed, it is not hard to see that one can replace $ X = \sum_i  x^i\frac{\p}{\p x^i} $ in \eqref{eq-mass-AF} by
$ X = r \nu_g$ (cf. \cite{MiaoTamXie16}) to obtain
\be \label{eq-mass-AF-2}
\begin{split}
\mathfrak{m}_{_{ADM}}
=  & \  c_n  \lim_{r\to\infty}   \int_{S_r} r \lf[ \Ric (g) - \frac12 \mathcal{S}_g \ri]  ( \nu_g ,\nu_g) d\gamma \\
=  & \  c_n  \lim_{r\to\infty}  \int_{S_r}  r ( H_2 - H_2^0)  d\gamma ,
\end{split}
\ee
where one has also used    the Gauss equation and $H_2$, $H_2^0$ denote  the second order mean curvature of
$S_r$ in $(M, g)$, $ \R^n$, respectively. (Here one assumes
that  $S_r $ can be isometrically embedded in $ \R^n$ which is  always satisfied  when $n=3$).
We recall that the second order mean curvature $H_2$  of  a hyersurface $S $ in a Riemannian manifold $(M, g)$
 is given by
\be\label{e-2ndmean-def}
H_2:=\sum_{1 \le \a<\beta  \le n-1}\kappa_\a\kappa_\beta=\frac12\lf(H^2-|A|^2\ri),
\ee
where $ \{ \kappa_\a \}_{1 \le \a \le n-1}$, $ A$ and $H$
are  the principal curvatures, the second fundamental form and the (first order) mean curvature
of $ S $ in $(M, g)$, respectively.
For a generalization of \eqref{eq-mass-AF-2}, we refer readers to Theorem 1.2 in  \cite{MiaoTamXie16}.

Motivated by \eqref{eq-mass-AF-2},  given a nontrivial function $ \phi $ on $ \Sigma = \p \Omega$, we
define
 \be
\mathcal{E}_\phi(g) : =\int_\Sigma \phi\, H_2(g) d\gamma ,
\ee
where $ H_2 (g) $ is the second order mean curvature of $\Sigma$ in $(\Omega, g)$ with resect to $\nu_g$.
Regarding $\mathcal{E}_\phi (\cdot)$, we have


\begin{thm} \label{t-main-2}
Let $(\Omega,g_0)$ and    $ \mathcal{M}^\lambda_\gamma $
be described as above.  Let $H_{g_0}$ be the mean curvature of $\Sigma$ in $(\Omega, g_0)$.
Let $ \{ g(t) \}_{| t | < \epsilon  }$ denote  a smooth path of metrics on $\Omega$ such that $ g(0) = g_0$.
Then
$$ \frac{d}{dt} \mathcal{E}_\phi ( g(t) ) |_{t=0} = 0  \ \ \text{for all $\{ g(t) \} \subset \mathcal{M}^\lambda_\gamma $} $$
if and only if
 the following conditions hold
\begin{enumerate}
  \item [(i)] $\phi\, H_{g_0}$ is the boundary value of some $N\in \mathrm{Ker}(D\mathcal{S}^*_{g_0})$.
  \item [(ii)] $\Sigma$ is umbilical at every point $p\in \Sigma$ where $\phi(p)\neq0$.
\end{enumerate}
\end{thm}

We give some  remarks about Theorem \ref{t-main-2}.

\begin{remark} \label{rmk-mt08}
Besides \eqref{eq-mass-AF-2}, our motivation to study $\mathcal{E}_\phi(\cdot)$ on $\mathcal{M}^\lambda_\gamma$
also comes from a corresponding result on  a  weighted total mean curvature functional  in \cite{MiaoShiTam09}.
More precisely,  if we let
$$
\mathcal{E}^{(1)}_\phi(g)=\int_\Sigma \phi H(g) d \gamma_g ,
$$
where $H(g) $ is the usual mean curvature of $\Sigma$ in $(\Omega,g)$,   it was proved in \cite{MiaoShiTam09} that
  $g_0$ is a critical point of $\mathcal{E}^{(1)} _\phi (\cdot) $ in  $\mathcal{M}_\gamma^\lambda$
  if and only if $\phi$ is the boundary value of some element $N \in \text{Ker}(D\mathcal{S}^*_{g_0})$.
\end{remark}

\begin{remark}
If $(\Omega, g_0)$ has a conformal Killing vector field $X$ such that
$ X$ does not vanish along $ \Sigma$ and is  orthogonal to $\Sigma$,
 then every point in $\Sigma$ is umbilical (see Proposition \ref{p-2ndmean} in Section 4).
\end{remark}

This paper is organized as follows. In Section \ref{s-conformal}, we discuss properties
 of manifolds of constant scalar curvature admitting a conformal Killing vector field.
In Section \ref{s-F}, we prove Theorem \ref{t-main-1}. In Section \ref{s-meancurvature}, we prove Theorem \ref{t-main-2}.

\section{Manifolds  with   a conformal Killing vector field} \label{s-conformal}


Let $ \Omega$, $ \Sigma$ be given as in the introduction.
Let $g_0$ be a  Riemannian metric on $ \Omega$ satisfying

\begin{enumerate}
  \item [{\bf (a1):}] $g_0$ has constant scalar curvature $\lambda n(n-1)$, with $\lambda=0, 1$ or $-1$,
  where $n$ is the dimension of $\Omega$;
  \item  [{\bf (a2):}] $(\Omega, g_0)$ has a conformal Killing vector field $X$, i.e.
  $$  \eta_{i |  j}+\eta_{j | i}=\frac2n (\div_{g_0}X)(g_0)_{ij}  $$
  where $\eta$ is the 1-form dual to $X$ with respect to $g_0$ and ``  ${ }_| $  " denotes the covariant differentiation
  on $(\Omega,g_0)$.
\end{enumerate}


\begin{lma}  \label{lma-Phi}
Suppose $(\Omega, g_0)$ satisfies  conditions {\bf (a1)}, {\bf (a2)}. Then
\begin{enumerate}
\item[(i)]  $ \tr_{g_0}( D \mathcal{S}_{g_0}^* ( \div_{g_0} X ))= 0 $, where ``$ \tr_{g_0} (\cdot)$"
denotes the trace  with respect to $g_0$.

\vspace{.2cm}

\item [(ii)]
$  D \mathcal{S}_{g_0}^* ( \div_{g_0} X )=-\frac{n}{n-2}\lf[ L_XG^{g_0}_\lambda +\frac{n-2}{n} (\div_{g_0} X )G^{g_0}_\lambda \ri],
$
where ``$L_X$" denotes  the Lie derivative along  $X$.
\end{enumerate}
\end{lma}

\begin{proof} For simplicity, we let  $f =\div_{g_0}X$ and denote  $G^{g_0}_\lambda $ by $G$.
For  any  smooth symmetric $(0,2)$ tensor $h$, computing in a local orthonormal frame $ \{ e_i \}_{ 1 \le i \le n }$,
we have
\bee
 \begin{split}
 & \ \la L_XG,h\ra-\la G,L_Xh\ra \\
 = & \ \la \nabla_XG ,h\ra-\la G ,\nabla_Xh\ra +\lf( G_{kj}\eta_{k|i}+G_{ik}\eta_{k|j}\ri)h_{ij}-\lf( h_{kj}\eta_{k|i}+h_{ik}\eta_{k|j}\ri)G_{ij}\\
 =& \ 2 \la  \nabla_XG,h\ra-X\la G,h\ra+h_{ij}\lf(G_{kj}\eta_{k|i}+G_{ik}\eta_{k|j}\ri)
 - h_{ij}\eta_{i|k}G_{kj}- h_{ij}\eta_{j|k} G_{ik}\\
 =& \ 2\la  \nabla_XG,h\ra-X\la G,h\ra+  h_{ij} \lf[  2 \lf(G_{kj}\eta_{k|i}+G_{ik}\eta_{k|j}\ri)
 -\frac4n f G_{ij}  \ri] \\
 =&\ 2\la L_XG, h\ra  -\frac4nf\la  G,h\ra   -X\la G,h \ra ,
 \end{split}
 \eee
 where
 summation is taken  over any pair of repeated indices and {\bf (a2)}  is applied.
  Therefore,
  \be \label{e-Lie-Co}
 \begin{split}
 \la L_XG,h\ra=-\la G,L_Xh\ra+\frac4nf\la  G,h\ra+X\la G,h \ra.
 \end{split}
 \ee
Now let  $h=g_0$. By \eqref{e-Lie-Co}, {\bf (a2)} and the fact $ \mathcal{S}_{g_0} = \l n (n-1)$,
\be\label{e-Lie-Co-1}
 \la G, g_0 \ra  = 0    \ \ \mathrm{and} \ \ \tr_{g_0}\lf(L_X G \ri) = 0.
\ee

To proceed, we adopt an argument of Herzlich  \cite{Herzlich}.
Let $ \{ \phi_t \}_{ | t | < \epsilon} $ be the $1$-parameter family of local diffeomorphisms  generated by  $X$.
Let $g_t=\phi_t^*(g_0)$,
then $  g_t  =e^{  2u_t } g_0$
for some smooth function $ u(\cdot, t) = u_t (\cdot)$  with $u_0=0$. Let $v=\frac{\p u }{\p t}|_{t=0}$.
By {\bf (a2)} and the fact
$ L_X g_0 = \frac{d  }{dt} g_t |_{t=0} $,
\be \label{e-divx-v}
f   = n v.
\ee
On the other hand,
\be
\begin{split}
\Ric (g_t) = & \ \Ric(g_0)-(n-2)(\nabla^2_{g_0}u_t-du_t\otimes du_t) \\
& \ -\lf(\Delta_{g_0}u_t+(n-2)|du_t|^2_{g_0}\ri) g_0 .
\end{split}
\ee
Taking the $t$-derivative  and letting  $t=0$, we have
\be \label{e-Lie-Ric}
\begin{split}
L_X \Ric(g_0)  =-(n-2)\nabla^2_{g_0}  v-(\Delta_{g_0}  v)g_0.
\end{split}
\ee
Hence, by \eqref{e-divx-v} and  \eqref{e-Lie-Ric},
\be \label{eq-Lie-end}
\begin{split}
& \ L_XG+\frac{n-2}{n} f G \\
= & \  - (n-2) D \mathcal{S}_{g_0}^* ( v) - (n-1) \lf(  \Delta_{g_0} v + \l n v \ri) g_0 \\
= & \  - \frac{n-2}n D \mathcal{S}_{g_0}^* (f ) + \frac1n\tr_{g_0} \lf( D \mathcal{S}_{g_0}^*(f) \ri) g_0.
\end{split}
\ee
Taking trace of \eqref{eq-Lie-end} and using \eqref{e-Lie-Co-1}, we conclude that (i) and (ii) of the lemma hold.
\end{proof}

If  $ g_0$ is Einstein,  then $ G^{g_0}_\lambda = 0 $,  which
implies $  D \mathcal{S}_{g_0}^* (  \div_{g_0} X )= 0  $ by Lemma \ref{lma-Phi} (ii).
(This was   the content of   \cite[Lemma 2.2]{Herzlich}.)
Next we show that the reverse is true under certain boundary and interior conditions.

\begin{prop}\label{p-Einstein}
Let  $(\Omega, g_0)$ satisfying  conditions  {\bf (a1),  (a2)}. Suppose
   \begin{enumerate}
     \item [(i)] $  D \mathcal{S}_{g_0}^* (  \div_{g_0} X ) = 0 $,
     \item [(ii)] $\la X,\nu_{g_0}\ra |G^{g_0}|^2\ge 0$ along $\Sigma$, and
     \item [(iii)] there exists  a smooth function $u$ on $ \Omega$
      such that
      $ \div_{g_0}X > X(u) $
      on a dense subset of $\Omega$.
   \end{enumerate}
     Then $g_0$ is Einstein in $\Omega$.
\end{prop}

\begin{proof}  Denote $G^{g_0}_\lambda $ by $G$.  By (i) and  Lemma \ref{lma-Phi},
$$  L_XG+\frac{n-2}{n} (\div_{g_0} X )G = 0  . $$
Therefore,
\be \label{eq-pde-G}
\begin{split}
\frac12X(|G|^2)= & \ \la \nabla_XG,G\ra\\
=& \ \la L_XG,G\ra-G_{kj}\eta_{k|i}G_{ij}-G_{ki}\eta_{k|j}G_{ij}\\
=& \ -\frac{n-2}n ( \div_{g_0} X)|G|^2-\frac2n ( \div_{g_0} X)|G|^2\\
= & \ - ( \div_{g_0} X)|G|^2.
\end{split}
\ee
 Multiplying \eqref{eq-pde-G} by a function $\phi$ and integrating on $ \Omega$, we have
\bee
\begin{split}
& \ - 2\int_\Omega  \phi (\div_{g_0}X)|G|^2dV_{g_0} \\
=& \ \int_\Omega   \lf[ X( \phi |G|^2) -X(\phi)|G|^2 \ri] dV_{g_0}\\
=& \ \int_\Omega   \lf[ \div_{g_0} ( \phi |G|^2X)- \phi \div_{g_0}(X)|G|^2 -X(\phi )|G|^2\ri] dV_{g_0} ,
\end{split}
\eee
where $ d V_{g_0}$ denotes the volume element of $g_0$.
Hence,
$$
\int_\Omega   \lf[   X(\phi )  - \phi  \div_{g_0}(X)\ri]  |G|^2 dV_{g_0}=\int_\Sigma \phi |G|^2 \la X,\nu_{g_0}\ra d\gamma_{g_0}  .
$$
Letting  $ \phi =e^{u}$, by (ii) we have
$$
\int_\Omega  e^{u} \lf[  X(u)-\div_{g_0}(X)   \ri] |G|^2dV_{g_0}\ge 0.
$$
By (iii),  we conclude that $G\equiv0$ and $g_0$ is Einstein.
\end{proof}

\begin{remark}
Proposition \ref{p-Einstein} implies that,  if $  \div_{g_0} X \in \mathrm{Ker} ( D \mathcal{S}_{g_0}^*) $
with $ \la X, \nu_{g_0} \ra \ge 0 $ along $ \Sigma$ and $ \div_{g_0} X > 0 $ in $\Omega$,
then $ g_0$ is Einstein.
\end{remark}

\section{Critical points of $\mathcal{F}_X( \cdot )$}\label{s-F}

Let  $g_0$ be a metric on $ \Omega$ satisfying conditions {\bf (a1),  (a2)}.
In this section, we consider the functional $\mathcal{F}_X (\cdot)$
near  $g_0$.

\begin{lma}\label{lma-1stvar-1}
   Let $\{ g(t) \}_{| t | < \epsilon } $ be a smooth family of metrics on $\Omega$
    with $g(0)=g_0$ and $g'(0)=h$.
    \begin{enumerate}
      \item [(i)] The derivative of $\mathcal{F}_X (g(t))$ at $t = 0 $ is given by
   \bee
   \begin{split}
& \     \frac{d}{dt} \mathcal{F}_X (g(t))|_{t=0} \\
= & \ \frac{n-2}{2n}\int_\Omega \lf[ \la D\mathcal{S}^*_{g_0}(\div_{g_0}X),h\ra_{g_0}  -   ( \div_{g_0} X )D \mathcal{S}_{g_0}(h)
     \ri] dV_{g_0}\\
&+\frac12\int_{\Sigma}\la G^{g_0}_\lambda, h\ra_{g_0} \la X,\nu_{g_0} \ra_{g_0}d\gamma_{g_0},
 \end{split}
\eee
 where $ D\mathcal{S}_{g_0}(\cdot) $  is the linearization of the scalar curvature  at $g_0$.

\vspace{.2cm}

 \item [(ii)]
If  $h|_{T(\Sigma)}=0$,
then
  \bee
  \begin{split}
 & \    \frac{d}{dt} \mathcal{F}_X (g(t))|_{t=0} \\
=  & \    \int_\Sigma   \lf[   \frac{n-2}{n} ( \div_{g_0} X )  D H_{g_0} (h)
 + \frac12 \la G^{g_0}_\lambda ,h\ra_{g_0} \la X,\nu_{g_0} \ra_{g_0} \ri] d\gamma_{g_0}  ,
 \end{split}
\eee
where $ D H_{g_0}(\cdot) $ is the linearization of the mean curvature $H$ of $\Sigma$  at $g_0$.

%
    \end{enumerate}
   \end{lma}

 \begin{proof}
 Let $\eta(t)$ be the $1$-form dual to $X$ with respect to $g(t)$.
 Let `` $ ; $ " denote  the covariant derivative  on $(\Omega, g(t))$.
 Let $ d V_{g(t)}$ be the volume element of $g(t)$.  Denote $G^{g(t)}_\lambda $ by $G^t$.
By the contracted Bianchi identity,  $\div_{g(t)}G^t=0$. Hence, $ \mathcal{F}_X (g(t))$ can
be rewritten as
\be\label{e-1stvar-1}
\mathcal{F}_X (g(t))
=  \int_{\Omega}  \la G^{t},\beta^t \ra_{g(t)} dV_{g(t)}  ,
\ee
where $\beta^t_{ij}=\frac12 (\eta_{i;j}+\eta_{j;i})$.

At $t=0$,  {\bf (a1)} and {\bf (a2)} imply
\be \label{e-1stvar-2}
 \la G^0,\beta^0 \ra_{g_0} = 0
 \ee  and
\be  \label{e-1stvar-3}
 \frac{d}{dt} \la G^{ t},\beta^t \ra_{g(t)}
=  \frac1n  f \lf[ -2 \la G^0, h\ra_{g_0}+  \tr_{g_0}(\frac{d}{dt}G^t)\ri]
 +\la G^0,\frac{d}{dt}\beta^t \ra_{g_0} .
\ee
where $ f = \div_{g_0} X$.
Using  the fact
$$ \tr_{g(t)}  G^t  = \frac{n-2}{2} \lf[ - \mathcal{S}_{g(t)} + \l (n-1)n \ri]
\  \mathrm{and} \  \beta^t = \frac12 L_X g(t),  $$
one has
\be\label{e-1stvar-4}
\tr_{g_0}(\frac{d}{dt}G^t) |_{t=0} = - \frac{n-2}{2} D \mathcal{S}_{g_0} (h) + \la  G^0,h\ra_{g_0}
\ee
and
\be \label{e-1stvar-5}
\begin{split}
\la G^0,\frac{d}{dt}\beta^t \ra_{g_0}|_{t=0} = & \ \frac12\la G^0,L_X h  \ra_{g_0}  .
\end{split}
\ee
Therefore,  by \eqref{e-1stvar-1} -- \eqref{e-1stvar-5} and  \eqref{e-Lie-Co},
  we have
\be  \label{e-1stvar-6}
\begin{split}
& \   \frac{d}{dt}\mathcal{F}_X (g(t))|_{t=0} \\
 = \ & \ \int_\Omega \lf[ - \frac{n-2}{2n} f   D \mathcal{S}_{g_0}(h)
  +  \frac12 \la G^0,  L_X h -\frac2n f   h \ra_{g_0}\ri] dV_{g_0}\\
 =& \ \int_\Omega  \lf\{ - \frac{n-2}{2n} f   D \mathcal{S}_{g_0}(h)
 +   \frac12 \la -L_XG^0 + \frac{2-n}{n} f   G^0,h\ra_{g_0} \ri. \\
 & \ \lf. +  \frac12 \div_{g_0} \lf[  \la G^0, h \ra_{g_0 } X \ri] \ri\}
\end{split}
\ee
(i) now follows from \eqref{e-1stvar-6} and  Lemma \ref{lma-Phi}(ii).
When $h|_{T(\Sigma)}=0$,
one easily checks that the following formula
\be
 \int_\Omega w D \mathcal{S}_{g_0} (h) d V_{g_0}
=  \int_\Omega \la D \mathcal{S}_{g_0}^* (w) , h \ra_{g_0} d V_{g_0}
- 2 \int_\Sigma w D H_{g_0} (h)  d \gamma_{g_0}
\ee
holds for any function $w$ on $ \Omega$ (cf. (34) in \cite{MiaoTam08} and (2.3) in \cite{MiaoShiTam09}).
Hence (ii) is true.
 \end{proof}


\begin{proof}[Proof of Theorem \ref{t-main-1}]

(i) For a variation $ \{ g(t) \} \subset \mathcal{M}^\lambda_0$, one has  $h = 0 $ at $ \Sigma$ and
$ D \mathcal{S}_{g_0} (h) = 0  $ on $ \Omega$. Hence, by  Lemma \ref{lma-1stvar-1}(i),
  \be \label{e-pf-df-1}
     \frac{d}{dt} \mathcal{F}_X (g(t))|_{t=0} \\
=  \frac{n-2}{2n}\int_\Omega  \la D\mathcal{S}^*_{g_0}(\div_{g_0}X),h\ra_{g_0}   dV_{g_0}  .
\ee
Therefore, if $ \div_{g_0}X \in \text{Ker}(D\mathcal{S}^*_{g_0}) $, then
$ \frac{d}{dt} \mathcal{F}_X (g(t))|_{t=0} = 0 $.

To see the converse is true, let $\hat h$ be an arbitrary symmetric $(0, 2)$ tensor
on $ \Omega$.
Since the first Dirichlet eigenvalue of $ (n-1) \Delta_{g_0} + \mathcal{S}_{g_0}$ is assumed to be positive,
by  the proof of \cite[Proposition 1]{MiaoTam08},  one can find a positive function $u(x,t) = u_t (x)$,  with  $ x \in \Omega$
   and $| t| < \delta  $   for some small $\delta$,
   such  that  $u_0  = 1$ on $ \Omega$,
    $g(t)  : =u_t^{\frac 4{n-2}}(g_0+t \hat h)$  has the same constant scalar curvature as $g_0$
      and    $u_t=1$ at $\Sigma$ for all $t$.
   Such a variation $\{ g(t) \}$ satisfies
   \be \label{e-g-h}
    h: = g'(0)=\frac4{n-2} \frac{\p u}{\p t}|_{t=0}  g_0 +\hat h
    \ee
with   $ h = \hat h $ at $ \Sigma$.
To proceed, we can choose $ \hat h $ to have compact support in $ \Omega$ for instance.
Then, by \eqref{e-pf-df-1}, \eqref{e-g-h} and Lemma \ref{lma-Phi} (i),
  \bee \label{e-pf-df-2}
   0
=  \frac{n-2}{2n}\int_\Omega  \la D\mathcal{S}^*_{g_0}(\div_{g_0}X), \hat h\ra_{g_0}   dV_{g_0}  ,
\eee
which readily shows  $ D\mathcal{S}^*_{g_0}(\div_{g_0}X) = 0 $.

(ii)  Denote $G^{g_0}_\lambda$ by $G^0$ for simplicity.
If  $\div_{g_0}X\in \text{Ker}(D\mathcal{S}^*_{g_0} (\cdot) )$
   and $\nu_{g_0}$ is an eigenvector of  $\la X,\nu_{g_0}\ra G^{0}$ with zero eigenvalue everywhere at $\Sigma$, then
     $ \frac{d}{dt}\mathcal{F}_X (g(t))|_{t=0}=0$ by  Lemma \ref{lma-1stvar-1}(i) and the fact $ h |_{T (\Sigma)} = 0$.

     Conversely, suppose
     $ \frac{d}{dt}\mathcal{F}_X (g(t))|_{t=0}=0 $
     for all $  \{ g(t) \} \subset  \mathcal{M}_\gamma^\lambda$,
     then $\div_{g_0}X \in \text{Ker}(D\mathcal{S}^*_{g_0}) $ by (i).
      Now let $\hat h$ be any symmetric $(0,2)$ tensor with $ \hat h |_{T (\Sigma) } = 0 $.
      By  the same construction of $\{ g(t) \}$ used in (i),  we have
      \be \label{eq-pf-df-3}
 0  = \frac12\int_{\Sigma}\la G^{0}, \hat h\ra_{g_0} \la X,\nu_{g_0} \ra_{g_0}d\gamma_{g_0}.
\ee
     Let $Q=\la X,\nu_{g_0}\ra G^{0}$ and choose $\hat h$ such that
     $$\hat h(u,v)= Q(u,v)-Q(u^T,v^T)  \ \ \mathrm{at} \ \Sigma ,$$
     where $ u$, $v$ are tangent vectors to $ \Omega$
     at $ \Sigma$ and $ u^T$, $ v^T$ denote the project of $u$, $v$ to the tangent space to $ \Sigma$, respectively.
     With this  choice of $ \hat h$,
      \eqref{eq-pf-df-3}  implies
     $ Q (\nu_{g_0} , u) = 0 $
for any $ u$ tangent to $ \Omega$ at $\Sigma$. Therefore, $ \nu_{g_0}$ is an eigenvector of $Q$
with zero eigenvalue.

   (iii) If  $g_0$ is Einstein, then $G^0 =0$ and
   $\div_{g_0}X\in \text{Ker}(D\mathcal{S}^*_{g_0}  )$.  Hence $
   \frac{d}{dt}\mathcal{F}_X (g(t))|_{t=0}=0
   $  for all variation  $\{ g(t) \} \subset \mathcal{M}^\lambda$ by Lemma \ref{lma-1stvar-1}(i). Next, suppose
    $   \frac{d}{dt}\mathcal{F}_X (g(t))|_{t=0}=0 $  for all such  $\{ g(t) \} $.
   Then  $\div_{g_0}X\in \text{Ker}(D\mathcal{S}^*_{g_0} )$ by (i) and
   the same construction of $\{ g(t) \}$ used in (i) shows that
    \eqref{eq-pf-df-3} holds  for any symmetric $(0,2)$ tensor $\hat h$ on $ \Omega$.
Let $ \psi$ be any extension of the function
$ \la X,\nu_{g_0}\ra_{g_0}  $ to $ \Omega$.
Let $ \hat h = \psi G^0$,
it then follows from \eqref{eq-pf-df-3} that
 \bee
|G^{g_0}|  \la X,\nu_{g_0}\ra=0.
\eee
By Proposition \ref{p-Einstein}, we conclude that $g_0$ is Einstein.
\end{proof}

\section{Critical points of $\mathcal{E}_\phi (\cdot)$}  \label{s-meancurvature}

Let $g_0$ be a metric on $ \Omega$ satisfying  {\bf (a1)}.
Let $H_{g_0}$ be the mean curvature of $ \Sigma $ in $(\Omega, g_0)$.
In this section, we discuss the behavior of $\mathcal{E}_\phi (\cdot)$ near $g_0$
in $\mathcal{M}_\gamma^\lambda$.

 \begin{lma}\label{l-2ndmean-1}
Let $ \{ g(t) \}_{ | t| < \epsilon } $ be a  smooth path of metrics with $ g(0) = g_0$ and $ g(t) |_{T(\Sigma) } = \gamma$
for all $t$.
Let $ H(t)$, $ A(t)$ denote  the mean curvature, the second fundamental form of $ \Sigma$ in $(\Omega, g(t))$,
respectively.
Let $\ringA   (t) $ be   the traceless part of $A(t)$. Then
\bee
\frac{d}{dt}\mathcal{E}_\phi(g(t))|_{t=0}=\int_\Sigma \phi \lf[  \frac{n-2 }{n-1}H(0)H'(0)-\la A'(0),
{\ringA (0)}   \ra_{\gamma}  \ri] d \gamma_{g_0} .
\eee
\end{lma}
\begin{proof}
Since $2H_2=H^2-|A|^2$,
\bee
\begin{split}
\frac{d}{dt}\mathcal{E}_\phi(g(t))|_{t=0}=& \ \int_\Sigma \phi \lf[ H(0)H'(0)-\la A'(0), A   \ra_\gamma  \ri] \dggz \\
=& \ \int_\Sigma \phi \lf[ H(0)H'(0)-\frac{H(0)}{n-1}\la A'(0),\gamma \ra_\gamma  -\la A'(0),\ringA (0) \ra_\gamma \ri] \dggz \\
=& \ \int_\Sigma \phi\lf[ \frac{n-2}{n-1}H(0)H'(0)  -\la A'(0),\ringA (0) \ra_\gamma  \ri]  \dggz ,
\end{split}
\eee
where we also have used the fact  $g(t) |_{T (\Sigma) } = \gamma$.
\end{proof}

 \begin{lma}\label{l-2ndmean-2}
Let $ \{ g(t) \}_{ | t| < \epsilon } $ be a  smooth path of metrics with $ g(0) = g_0$.
Let $ h = g'(0)$.
   Let $\{x^1,\dots, x^{n-1}, x^n \} $ be local coordinates of $\Omega$ near  $\Sigma$ such that,
   when $x^n=0$, $\{ x^1, \ldots, x^{n-1} \}$ form local coordinates on $\Sigma$ and $\p_{x^n}=\nu_{g_0}$.
Then
\be \label{eq-dAt}
 A'_{\alpha \beta} (0)  =  - \frac12  \lf[ h_{n \beta; \alpha} + h_{\alpha n; \beta} - h _{\alpha \beta; n} \ri]
 +  \frac12 A_{\alpha \beta}(0) h_{nn} ,
\ee
where $1\le \a,\beta\le n-1$.
\end{lma}
 \begin{proof} The outward unit normal $\nu_{g(t)} $ of $\Sigma$ in $(\Omega, g(t)) $ is given by
 $$
 \nu_{g_t} =\frac{\nabla_{g(t)}  x^n}{|\nabla_{g (t)} x^n|_{g(t)} }=\frac{g^{ni}} {\sqrt{g^{nn}}}\p_{x_i}.
 $$
 Let $\Gamma_{ij}^k$, $\nabla^t$ denote the Christoffel symbols, the
 covariant differentiation of $g(t)$, respectively.
 Then
 \bee
 A_{\a\beta} (t)
 = -\la \nu_{g(t)} ,\nabla^t_{\p_{x^\a}}\p_{x^\beta}\ra_{g(t)}
 = -\frac1{\sqrt{g^{nn}}}\Gamma_{\a\beta}^n .
 \eee
Therefore,
 \bee
 \begin{split}
 A_{\a\beta}'(0)=&-\frac12\frac1{\sqrt{g^{nn}}}g^{nj}\lf(h_{\a j;\beta}+h_{j\beta;\a}-h_{\a\beta;j}\ri)-\frac12\frac{g^{nj}g^{in}h_{ij}}{(g^{nn})^\frac32}\Gamma_{\a\beta}^n(0)\\
 =&-\frac12\lf(h_{\a n;\beta}+h_{n\beta;\a}-h_{\a\beta;n}\ri)+\frac12h_{nn}A_{\a\beta},
 \end{split}
 \eee
which gives \eqref{eq-dAt}.
 \end{proof}

\begin{prop}\label{p-2ndmean-1}
Suppose  $\phi$ is  a nontrivial function on $\Sigma$ such that
\begin{enumerate}
 \item [(i)] $\phi\, H_{g_0}$ is the boundary value of some $N\in \text{Ker}(D\mathcal{S}^*_{g_0})$;
  \item [(ii)] $\Sigma$ is umbilical in $ (\Omega, g_0)$ at every point $p\in \Sigma$ where $\phi(p)\neq0$.
\end{enumerate}
Then $ \frac{d}{dt}\mathcal{E}_\phi(g(t))|_{t=0} = 0 $ for any path $ \{ g(t) \} \subset \mathcal{M}^\lambda_\gamma$
with $ g(0) = g_0$.
\end{prop}
\begin{proof}
Condition  (ii) simply means $\ring A (0) = 0 $ wherever $\phi \neq 0$ along $ \Sigma$.
Hence, by Lemma \ref{l-2ndmean-1},
\bee
\begin{split}
\frac{d}{dt}\mathcal{E}_\phi(g(t))|_{t=0}
= & \  \frac{n-2}{n-1} \int_\Sigma \phi H_{g_0} H'(0)  \dggz \\
= & \  \frac{n-2}{n-1}  \frac{d}{dt}\mathcal{E}^{(1)}_{ \phi H_{g_0} } (g(t))|_{t=0} ,
\end{split}
\eee
where the functional  $ \mathcal{E}^{(1)}_{(\cdot)} (\cdot)$ is defined in Remark \ref{rmk-mt08}.
But
$$ \frac{d}{dt}\mathcal{E}^{(1)}_{ \phi H_{g_0} } (g(t))|_{t=0} = 0$$
by condition (i) and  \cite[Th. 2.1]{MiaoShiTam09}.
Hence the result follows.
\end{proof}

Next we want to show that the converse of  Proposition \ref{p-2ndmean-1} is also true.
This is proved in the next two lemmas.

\begin{lma}\label{l-2ndmean-3}
If $ \frac{d}{dt}\mathcal{E}_\phi(g(t))|_{t=0} = 0 $ for any path $ \{ g(t) \} \subset \mathcal{M}^\lambda_\gamma$
with $ g(0) = g_0$,
then  $\phi H_{g_0}$ is the boundary value of a function  in $ \mathrm{Ker}(D\mathcal{S}^*_{g_0})$.
\end{lma}
\begin{proof}
Let $\hat h$ be any  symmetric $(0,2)$  tensor with compact support in $\Omega$.
By \cite[Proposition 1]{MiaoTam08},  one can find a positive function $u(x,t) = u_t (x)$,  with  $ x \in \Omega$
   and $| t| < \delta  $   for some small $\delta$,
   such  that  $u_0  = 1$ on $ \Omega$, $u_t = 1 $ at $ \Sigma$ and
    $g(t)  : =u_t^{\frac 4{n-2}}(g_0+t \hat h) \in \mathcal{M}^\lambda_\gamma$.
   Let $ h = g'(0)$, then
   $ h=\frac4{n-2} v   g_0 +\hat h$, where $ v  = \frac{\p u}{\p t}|_{t=0}$.
 By Lemma \ref{l-2ndmean-1},  this path $\{g(t)\}$ satisfies
\be \label{eq-var-E}
\int_\Sigma \phi \lf[  \frac{n-2 }{n-1}H_{g_0} H'(0)-\la A'(0),
{\ringA (0)}   \ra_{\gamma}  \ri] d \gamma_{g_0} =0  .
\ee
We claim that, for such a variation, one indeed has
\be\label{e-2ndmean-1}
 \la A'(0), {\ringA (0)}\ra_\gamma  =0.
\ee
If  this is true, then \eqref{eq-var-E} shows
$  \frac{d}{dt}\mathcal{E}^{(1)}_{ \phi H_{g_0} } (g(t))|_{t=0} = 0 $
for this  path $ \{ g(t) \} $.
Therefore, by the same proof of \cite[Th. 2.1]{MiaoShiTam09}, the result follows.

To prove \eqref{e-2ndmean-1}, we note that, along $ \Sigma$,
\bee
h_{n\beta;\a}=\frac 4{n-2} v_{,\a} {g_0}_{n\beta}+\hat h_{n\beta;\a}=0
\eee
because $\hat h$ has compact support and ${g_0}_{n\beta}=0$. Here ``," denotes partial derivative.
 Similarly
\bee
h_{\a\beta;n}=\frac 4{n-2}v_{,n} {g_0}_{\a\beta}+\hat h_{\a\beta;n}=\frac 4{n-2}v_{,n} \gamma_{\a\beta},
\eee
\bee
h_{nn}=\frac 4{n-2} v {g_0}_{nn}+\hat h_{nn}=0
\eee
because $v =0$ at $\Sigma$. By Lemma \ref{l-2ndmean-2}, we have
$$
A_{\a\beta}'(0)=\frac 2{n-2}v_{,n} \gamma_{\a\beta}.
$$
Thus,
\bee
\la A'(0), \ringA  (0) \ra_\gamma =\frac 2{n-2} v_{,n} \la \gamma,\ringA (0)\ra_\gamma =0
\eee
as  $\ringA (0)$ is traceless.  Hence, \eqref{e-2ndmean-1} is true. This completes the proof.
\end{proof}

\begin{lma}\label{l-2ndmean-4}
Suppose $ \frac{d}{dt}\mathcal{E}_\phi(g(t))|_{t=0} = 0 $ for any path $ \{ g(t) \} \subset \mathcal{M}^\lambda_\gamma$
with $ g(0) = g_0$,
Then $\Sigma$ is umbilical at every point $p\in \Sigma$ with $\phi(p)\neq 0$.
\end{lma}
\begin{proof}
By Lemma \ref{l-2ndmean-3}, $\phi H_{g_0}$ is the boundary value of a function in $ \mathrm{Ker} (D \mathcal{S}^*_{g_0} )$.
Hence,  by \cite[Th. 2.1]{MiaoShiTam09},
\bee
\int_\Sigma \phi H_{g_0} H'(0) d \gamma_{g_0} =0
\eee
for any variation $\{ g(t) \} \subset \mathcal{M}_\gamma^K$.
Therefore, by Lemma \ref{l-2ndmean-1},
\be \label{eq-pf-df-5}
\int_\Sigma \phi \la A'(0), {\ringA} (0)    \ra_\gamma  \dggz  = 0
\ee
for all such paths.

Now let $p\in \Sigma$ with $\phi(p) > 0$ for instance.
Let $\eta\ge 0$ be a smooth function with support near $p$ such that $\eta(p)=1$ and $\eta(q)>0$ implies $\phi(q)>0$.
Let $\hat h$ be a symmetric $(0,2)$ tensor on $ \Omega$
such that $ \hat h |_{T (\Sigma) } = 0 $, $ \hat h( \nu_{g_0}, w) = 0 $ for all $ w$ tangential to $ \Sigma$
and  $\hat h (\nu_{g_0}, \nu_{g_0} ) = 1$
along $ \Sigma$.
By \cite[Proposition 1]{MiaoTam08},  there exists a positive function $u(x,t) = u_t (x)$,  with  $ x \in \Omega$
   and $| t| < \delta  $   for some small $\delta$,
   such  that  $u_0  = 1$ on $ \Omega$, $u_t = 1 $ at $ \Sigma$ and
    $$g(t)  : =u_t^{\frac 4{n-2}}(g_0+t  \eta \hat h) \in \mathcal{M}^\lambda_\gamma . $$
   Let $ h = g'(0)$, then
   $ h=\frac4{n-2} v   g_0 + \eta \hat h$, where $ v  = \frac{\p u}{\p t}|_{t=0}$.
At $ \Sigma$, we have
\be \label{e-2ndmean-2}
h_{n\beta;\a}= \eta \hat h_{n\beta;\a}
= -\eta\lf(\Gamma_{\a n}^k \hat h_{k\beta}+\Gamma_{\a\beta}^k\hat h_{nk}\ri)
= \eta A_{\a\beta} (0) ,
\ee
\be\label{e-2ndmean-3}
\begin{split}
h_{\a\beta;n}=&\frac{4  }{n-2} v_{,n} \gamma_{\a\beta}+\eta_{,n} \hat h_{\a\beta}+\eta\hat h_{\a\beta;n} \\
=& \frac{4  }{n-2} v_{,n} \gamma_{\a\beta}- \eta \lf(\Gamma_{n\a  }^k \hat h_{k\beta}+\Gamma_{\beta n }^k\hat h_{\a k}\ri)\\
=& \ \frac{4  }{n-2} v_{,n} \gamma_{\a\beta},
\end{split}
\ee
\be\label{e-2ndmean-4}
h_{nn}=\frac4{n-2} v +\eta =\eta.
\ee
Therefore, by Lemma \ref{l-2ndmean-2} and \eqref{e-2ndmean-2} -- \eqref{e-2ndmean-4},
\bee
\begin{split}
A'_{\a\beta}(0)=&-\frac12\lf[ 2 \eta A_{\a\beta} (0) -  \frac{4  }{n-2} v_{,n} \gamma_{\a\beta}\ri] +\frac12\eta A_{\a\beta}(0) \\
=&-\frac12\lf[  \eta A_{\a\beta} (0) -  \frac4{n-2} v_{,n} \gamma_{\a\beta}\ri] .
\end{split}
\eee
Plugging it to \eqref{eq-pf-df-5} gives
\bee
\begin{split}
0 =& \ \int_\Sigma\phi\la  \lf(   \eta A -  \frac{4 }{n-2}v_{,n} \gamma \ri),\ringA(0) \ra_\gamma  \dggz \\
=& \ \int_\Sigma\phi\eta\la  \ringA   ,\ringA\ra_\gamma  \dggz .
\end{split}
\eee
Since $\phi\eta\ge0$ and $\eta(p)\phi(p)>0$. We conclude  $\ringA(p)=0$.
\end{proof}

Theorem \ref{t-main-2} in the introduction now follows directly from Proposition \ref{p-2ndmean-1},
 Lemma \ref{l-2ndmean-3} and  Lemma \ref{l-2ndmean-4}.

We end this paper  with a proposition  relating Theorems \ref{t-main-1} and \ref{t-main-2}, followed by an example.

\begin{prop}\label{p-2ndmean}
Let $(\Omega^n, g_0)$ be a compact Riemannian manifold with  boundary $\Sigma$.
Suppose $X$ is a conformal Killing vector field on $(\Omega, g_0)$ such that $X$ does not
vanish along  $\Sigma$ and  is {normal} to $\Sigma$. Then every point in $\Sigma$ is umbilical and the mean curvature  $H_{g_0}$ of $\Sigma$ satisfies
\be \label{eq-H-final}
H_{g_0}=\frac{n-1}{n} \frac{\div_{g_0}X}{{ \la X,\nu_{g_0} \ra   }},
\ee
As a result,  if $\mathcal{S}(g_0)=\lambda n(n-1)$  and if
$ \frac{d}{dt} \mathcal{F}_X ( g(t) )|_{t=0} = 0 $ for all variation $\{ g(t) \}_{|t|< \epsilon} \subset \mathcal{M}_0^\lambda$,
then
$ \frac{d}{dt} \mathcal{E}_{{\phi}} ( g(t) )|_{t=0} = 0 $
for all variation $\{ g(t) \}_{|t|< \epsilon} \subset \mathcal{M}_\gamma^\lambda$, {where $\phi=\la X,\nu_0\ra$}.
\end{prop}

\begin{proof}
 By the assumption on $ X$,   $\nu_{g_0} =X/{\la X,\nu_{g_0} \ra }$ at $ \Sigma$.  Let $\{ e_i \}_{1 \le i \le n} $
be a local orthonormal frame with $e_n=\nu_{g_0}$ at points in $ \Sigma$.
Then the second fundamental form $A$ of $ \Sigma$  is given by
\bee
A_{\a\beta}
=  \la \nabla_{e_\a}\nu_{g_0} ,e_\beta\ra
=  \frac1{{\la X,\nu_{g_0} \ra}} \la \nabla_{e_\a}X,e_\beta\ra
=  \frac1{{\la X,\nu_{g_0} \ra}}\eta_{\beta|\a}
\eee
for $1\le \a, \beta\le n-1$, where  $\eta$ is the $1$-form dual to $X$ on  $(\Omega, g_0)$.
Hence,
$$
A_{\a\beta}=\frac12(A_{\a\beta}+A_{\beta\a})=\frac1{2{ \la X,\nu_{g_0} \ra}}(\eta_{\a|\beta}+\eta_{\beta|\a})=
\frac{\div_{g_0}X} {n{ \la X,\nu_{g_0} \ra}}\gamma_{\a\beta},
$$
which  proves the first statement of the proposition.

Next, suppose  $\mathcal{S}(g_0)=\lambda n(n-1)$  and
$ \frac{d}{dt} \mathcal{F}_X ( g(t) )|_{t=0} = 0 $ for all variation $\{ g(t) \}_{|t|< \epsilon} \subset \mathcal{M}_0^\lambda$.
Then, by Theorem \ref{t-main-1} (i), $ \div_{g_0} X \in \mathrm{Ker}( D \mathcal{S}^*_{g_0} )$.
Hence, by  \eqref{eq-H-final} and  Theorem \ref{t-main-2},
$ \frac{d}{dt} \mathcal{E}_{\phi} ( g(t) )|_{t=0} = 0 $ for all variation $\{ g(t) \}_{|t|< \epsilon} \subset \mathcal{M}_\gamma^\lambda$.
\end{proof}

\vh

\noindent {\bf Example}
{
Consider a warped product $ (\Omega^n, g) = ( I \times N , d r^2 + f^2(r) \bar g) $
where $I \subset \R^1$ is a closed  interval, $(N, \bar g)$ is a closed Riemannian  manifold $N$,
and $f$ is a smooth positive  function on $I$.
  Consider a vector field  $X=\phi(r)\p_r$, where $ \phi $ is any smooth function on $I$.
  The $1$-form dual to $X$ is   $\eta=\phi(r)dr$.
  Let $\{ x^i\}_{ 1\le i\le n-1}$ be local coordinates of $N$ and let $x^n=r$.
   The Christoffel  symbols $ \{ \Gamma_{ij}^n \}$  of $g$ are given by
 \bee
  \left\{
    \begin{array}{rlll}
      \Gamma_{ij}^n & = &   -ff'\bar g_{ij}, & \hbox{\ if $1\le i, j\le n-1$;} \\
     \Gamma_{ni}^n & =  & 0,\ \ &\hbox{\ if $1\le i\le n-1$;} \\
     \Gamma_{nn}^n & =  &  0. &
    \end{array}
  \right.
 \eee
Since $\bar g_{ij}=f^{-2}g_{ij}$, we have
\bee
\eta_{i|j}=\frac{\p \eta_i}{\p x^j}-\Gamma_{ij}^k\eta_k=\frac{\p \eta_i}{\p x^j}-\Gamma_{ij}^n\phi
=\left\{
   \begin{array}{ll}
    \frac{f'}f \phi g_{ij},  & \hbox{if $1\le i, j\le n-1$;} \\
     \phi', & \hbox{if $i=j=n$;} \\
     0, & \hbox{otherwise.}
   \end{array}
 \right.
\eee
 From this, one concludes that $\eta_{i|j}+\eta_{i|j}=\frac2n (\div_g X) g_{ij} $, $1 \le i, j \le n$,  if and only if $\phi'=(f'\phi)/f$,
 which is equivalent to $\phi=cf$ for some constant $c$. Now set $\phi=f$,  then $\div_g X=nf'$.
 Note that this $X$ is never zero on $\p \Omega$  and is normal to $\p \Omega$.
 Classification for $f'$ being in $\text{Ker}(D\mathcal{S}*_g)$ can be found in \cite{Kobayashi,KobayashiObata}.
}

\end{document}